\documentclass{amsart}

\usepackage{macros}
\standardsettings

\newcommand{\OI}{[0,1]}

\begin{document}
\title[Diophantine properties of measures]{Diophantine properties of measures invariant with respect to the Gauss map}

\author{Lior Fishman}
\address{University of North Texas, Department of Mathematics, 1155
  Union Circle \#311430, Denton, TX 76203-5017, USA} 
\email{Lior.Fishman@unt.edu}

\author{David Simmons}
\address{University of North Texas, Department of Mathematics, 1155
  Union Circle \#311430, Denton, TX 76203-5017, USA} 
\email{DavidSimmons@my.unt.edu}

\author{Mariusz Urba\'nski}
\address{University of North Texas, Department of Mathematics, 1155
  Union Circle \#311430, Denton, TX 76203-5017, USA} 
\email{urbanski@unt.edu  \newline \hspace*{0.3cm} Web:\!\!\!
www.math.unt.edu/$\sim$urbanski}

%\subjclass[2010]{Primary }
%\keywords{}
%\date{}
%\dedicatory{}
\begin{abstract}
Motivated by the work of D. Y. Kleinbock, E. Lindenstrauss, G. A. Margulis, and B. Weiss \cite{KLW, KM}, we explore the Diophantine properties of probability measures invariant under the Gauss map. Specifically, we prove that every such measure which has finite Lyapunov exponent is extremal, i.e. gives zero measure to the set of very well approximable numbers. We show on the other hand that there exist examples where the Lyapunov exponent is infinite and the invariant measure is not extremal. Finally, we construct a family of Ahlfors regular measures and prove a Khinchine-type theorem for these measures. The series whose convergence or divergence is used to determine whether or not $\mu$-almost every point is $\psi$-approximable is different from the series used for Lebesgue measure, so this theorem answers in the negative a question posed by Kleinbock, Lindenstrauss, and Weiss \cite{KLW}.% In addition, the theorem proves the optimality of a result of Weiss \cite{Weiss}.
% We provide a partial converse to B. Weiss's result of Khinchine type \cite{Weiss} by constructing a large class of measures, which are both conformal and Ahlfors regular, and for which the divergence of Weiss's series entails the $\psi$-approximability of almost all numbers. As a result we answer in the negative a question 
\end{abstract}
\maketitle

\section{Introduction} \label{sectionintroduction}

\begin{definition}
Let $\psi:\N\rightarrow(0,\infty)$ be any function. We recall that an irrational $x\in\OI$ is \emph{$\psi$-approximable} if there exist infinitely many $p/q\in \Q$ such that
\begin{equation}
\label{approximable}
\left|x - \frac{p}{q}\right| \leq \psi(q).
\end{equation}

We recall the following facts and definitions from the classical theory of Diophantine approximation:
\begin{itemize}
\item Every $x$ is $\psi$-approximable when $\psi(q) = q^{-2}$.
\item $x$ is \emph{badly approximable} if there exists $\varepsilon > 0$ such that $x$ is not $\psi$-approximable when $\psi(q) = \varepsilon q^{-2}$. The set of badly approximable numbers has Hausdorff dimension one but Lebesgue measure zero.
\item $x$ is \emph{very well approximable} if there exists $c > 0$ such that $x$ is $\psi$-approximable when $\psi(q) = q^{-(2 + c)}$. The set of very well approximable numbers has Hausdorff dimension one but Lebesgue measure zero.
\item $x$ is a \emph{Liouville} number if for all $c > 0$ the number $x$ is $\psi$-approximable when $\psi(q) = q^{-c}$. The set of Liouville numbers has Hausdorff dimension zero.
\end{itemize}
\end{definition}

{\bf Convention.} The symbols $\lesssim$, $\gtrsim$, and $\asymp$ will denote multiplicative asymptotics. For example, $A\lesssim B$ means that there exists a constant $C > 0$ (the \emph{implied constant}), such that $A\leq C B$.

\subsection{Extremal measures}
A measure\footnote{In this paper all measures are assumed to be Borel and locally finite.} $\mu$ on $\R$ is said to be \emph{extremal} if the set of very well approximable numbers is null with respect to $\mu$. In other words, $\mu$ behaves like Lebesgue measure with respect to very well approximable numbers. This definition was introduced by D. Y. Kleinbock, E. Lindenstrauss, and B. Weiss in \cite{KLW}, as a generalization of the notion of an extremal manifold, which was defined by V. Sprind\v{z}uk. B. Weiss \cite{Weiss} proved that measures which satisfy a certain decay condition, called \emph{absolutely decaying}, are extremal.

\begin{definition}
For $\alpha > 0$, a measure $\mu$ on $\R$ is said to be \emph{absolutely $\alpha$-decaying} if there exists $C > 0$
such that for all $x\in\R$, for all $0 < r \leq 1$ and for all $0 <
\varepsilon \leq 1$ we have 
\begin{equation}
\label{decay}
\mu(B(x,\varepsilon r)) \leq C\varepsilon^\alpha \mu(B(x,r)).
\end{equation}
It is said to be \emph{absolutely decaying} if it is absolutely $\alpha$-decaying for some $\alpha > 0$.
\end{definition}
We recall also that for $\delta > 0$, a measure $\mu$ on $\R$ is \emph{Ahlfors $\delta$-regular} if there exist positive constants $C_1$ and $C_2$ such that
\[
C_1 r^\delta \leq \mu(B(x,r)) \leq C_2 r^\delta
\]
for all $x$ in the topological support of $\mu$ and for all $0 < r \leq 1$. Examples of Ahlfors regular measures include Lebesgue measure and the Hausdorff measure on certain fractals such as the Cantor set. Clearly, any Ahlfors $\delta$-regular measure on $\R$ is automatically absolutely $\delta$-decaying.

Generalizations of Weiss's result to higher dimensions have been considered by Kleinbock, Lindenstrauss, and Weiss \cite{KLW}. However, for the purposes of this paper we will consider only Weiss's original result and not the higher dimensional generalizations.

Let $G:\OI\rightarrow\OI$ be the Gauss map, i.e. 
\begin{equation}
\label{Gauss}
G(x)=\begin{cases}
\frac1{x}-\lfloor 1/x\rfloor & x > 0\\
0 & x = 0
\end{cases},
\end{equation}
where $\lfloor x\rfloor$ is the integer part of $x$. A measure $\mu$ is \emph{invariant} with respect to the Gauss map if $\mu\circ G^{-1} = \mu$. In this paper we consider the extremality of probability measures invariant with respect to the Gauss map. Specifically, we show that if an invariant measure $\mu$ has finite Lyapunov exponent, then $\mu$ is extremal.
\begin{definition}
If $\mu$ is a probability measure on $\OI$ invariant with respect to the Gauss map, then the integral
\[
\chi_\mu(G)=\int\log|G'|d\mu
\]
is called the \emph{Lyapunov exponent} of the measure $\mu$ with respect to the Gauss map $G$.
\end{definition}

\begin{reptheorem}{theoremextremal}
If $\mu$ is a probability measure on $\OI\butnot\Q$ invariant with respect to the Gauss map $G$ with finite Lyapunov exponent $\chi_\mu(G)$, then $\mu$ is extremal.
\end{reptheorem}

The assumption that $\chi_\mu(G) < \infty$ is a very reasonable assumption which is satisfied for a large class of dynamically defined measures; see Section~\ref{sectionconformalextremal}. In particular there exist measures which satisfy this assumption but are not absolutely decaying. It is also a necessary assumption, as seen from the following:

\begin{reptheorem}{theoremcounterexample}
There exists a measure $\mu$ invariant with respect to the Gauss map which gives full measure to the Liouville numbers. In particular, $\mu$ is not extremal.
\end{reptheorem}

\subsection{A question about absolutely decaying measures}
In \cite{KLW}, Kleinbock, Lindenstrauss, and Weiss asked the following question:\footnote{Actually, Kleinbock, Lindenstrauss, and Weiss's question was about friendly measures on $\R^d$. When restricted to one dimension, friendly measures are the same as absolutely decaying measures (see Lemma 2.2 in \cite{KLW}).}
\begin{question}[Question 10.1 of \cite{KLW}]\label{questionKLW}
Suppose that $\mu$ is an absolutely decaying measure on $\R$.
\begin{itemize}
\item[(a)] Is it true that for any decreasing function $\psi:\N\rightarrow(0,+\infty)$, either the set of $\psi$-approximable numbers or its complement has $\mu$-measure $0$?
\item[(b)] Is it true that for all $\psi$ as in \textup{(a)}, $\mu$-almost every point is $\psi$-approximable if and only if
\begin{equation}
\label{KLW}
\sum_{q = 1}^\infty q\psi(q) = \infty ?\footnote{In this formula we have replaced $\psi(q)$ by $q\psi(q)$ due to a difference in the definition of $\psi$-approximability.}
\end{equation}
\end{itemize}
\end{question}
We answer this question in the negative by constructing a family of measures on $\R$ which are Ahlfors regular (and in particular absolutely decaying) and yet do not satisfy either (a) or (b).

To construct these measures, we fix a set $I\subset\N$ and let
\[
J_I = \{x\in\OI\butnot\Q:\text{ the continued fraction entries of $x$ lie in $I$}\}.
\]
\begin{reptheorem}{theoremconversestrong}
Fix an infinite set $I\subset\N$, and let $h$ be the Hausdorff dimension of $J_I$. Assume that the $h$-dimensional Hausdorff measure $\HH^h$ restricted to $J_I$ is Ahlfors $h$-regular. Let $\mu = \HH^h\given_{J_I}$, and let $\psi:\N\rightarrow(0,+\infty)$ be a function such that the function $q\mapsto q^2\psi(q)$ is nonincreasing. Then $\mu$-almost every (resp. $\mu$-almost no) point is $\psi$-approximable, assuming that the series
\begin{equation}
\tag{\ref{weiss}}
\sum_{q = 1}^\infty q^{2\alpha - 1}\psi(q)^\alpha
\end{equation}
diverges (resp. converges).
\end{reptheorem}
We note that the convergence case of Theorem \ref{theoremconversestrong} is a theorem of Weiss \cite{Weiss}, which he proved for any absolutely decaying measure $\mu$ and for any function $\psi:\N\rightarrow(0,+\infty)$.

Note that when $I = \N$, then $\mu$ is Lebesgue measure on $[0,1]$, and Theorem \ref{theoremconversestrong} reduces to the classical Khinchine theorem.

It appears that the only easy example of a set $I$ satisfying the hypotheses of Theorem \ref{theoremconversestrong} is the set $I = \N$. Nevertheless we will demonstrate the following:

\begin{reptheorem}{theoremexistence}
For every $0 < \delta \leq 1$ there exists an infinite set $I\subset\N$ such that $\HD(J_I) = \delta$ and such that $\HH^\delta\given_{J_I}$ is Ahlfors $\delta$-regular.\footnote{Here and from now on $\HD(S)$ denotes the Hausdorff dimension of a set $S$. $\HH^\delta(S)$ and $\PP^\delta(S)$ denote its $\delta$-dimensional Hausdorff and packing measure, respectively.}
\end{reptheorem}

Combining Theorems \ref{theoremconversestrong} and \ref{theoremexistence} in the obvious way yields the following corollary:

\begin{corollary}
\label{corollaryconverse}
For every $0 < \alpha \leq 1$, there exists an Ahlfors $\alpha$-regular, and therefore absolutely $\alpha$-decaying, measure $\mu$ such that for any function $\psi:\N\rightarrow(0,+\infty)$ such that the function $q\mapsto q^2\psi(q)$ is nonincreasing, then $\mu$-almost every (resp. $\mu$-almost no) point is $\psi$-approximable, assuming that the series \textup{(\ref{weiss})} diverges (resp. converges).
\end{corollary}
In the case $\alpha = 1$, the measure is simply Lebesgue measure.

\begin{remark}
It appears that (when $\alpha < 1$) this is the first example of a measure $\mu$ which is neither atomic nor absolutely continuous to Lebesgue for which a complete criterion has been given for when the set of $\psi$-approximable numbers is $\mu$-null or $\mu$-full.
\end{remark}

\begin{corollary}
\label{corollarynegativeanswer}
The answer to Question \ref{questionKLW} is negative (for both parts \textup{(a)} and \textup{(b)}).
\end{corollary}
\begin{proof}
Fix $0 < \alpha < 1$ and let $\mu$ be the measure guaranteed by Corollary \ref{corollaryconverse}. To see that the answer to (b) is negative, we merely note the existence of a function $\psi$ for which (\ref{KLW}) converges but (\ref{weiss}) diverges. For example,
\[
\psi(q) = \frac{1}{q^2\log^{1/\alpha}(q)}\cdot
\]

To see that the answer to (a) is negative, let $y\in\R$ be chosen at random with respect to Lebesgue measure. As noted in \cite{KLW} (see the paragraph immediately following Question 10.1), the measure $\nu := \mu\circ (x\mapsto x + y)^{-1}$ does satisfy (b) of Question \ref{questionKLW}. But then the measure $\mu + \nu$ is also Ahlfors regular, but does not satisfy (a); indeed, for the function $\psi$ given above, $\mu$-almost every point but $\nu$-almost no point is $\psi$-approximable.
\end{proof}

\ignore{
\begin{remark}
In addition to answering Question \ref{questionKLW}, Corollary \ref{corollaryconverse} can also be viewed as demonstrating the optimality of Theorem \ref{theoremweiss}. Indeed, the theorem shows that the series (\ref{weiss}) cannot be replaced by any series which converges for a larger class of functions $\psi$, even if the hypothesis of absolute $\alpha$-decay is replaced by the stronger hypothesis of Ahlfors $\alpha$-regularity.
\end{remark}

Weiss's result mentioned above regarding the extremality of absolutely decaying measures is in fact a corollary of the following Khinchine-type theorem:

\begin{theorem}[Weiss \cite{Weiss}] \label{theoremweiss}
Fix $0 < \alpha \leq 1$, and suppose that $\mu$ is an $\alpha$-absolutely decaying measure on $\R$. If $\psi:\N\rightarrow(0,+\infty)$ and if the series
\begin{equation}
\label{weiss}
\sum_{q = 1}^\infty q^{2\alpha - 1}\psi(q)^\alpha
\end{equation}
converges, then $\mu$-almost every point of $\R$ is not $\psi$-approximable.
\end{theorem}
Note that when $\mu$ is Lebesgue measure and $\alpha = 1$, then this theorem corresponds to the convergence part of the classical Khinchine theorem.

Fix $0 < \alpha \leq 1$, a measure $\mu$, and a function $\psi:\N\rightarrow(0,\infty)$. We shall say that \emph{the converse to Theorem~\ref{theoremweiss} holds} for $\alpha$, $\mu$, and $\psi$ if the divergence of the series (\ref{weiss}) implies that $\mu$-almost every point of $\R$ is $\psi$-approximable. Then the divergence part of the classical Khinchine theorem says precisely that if $\alpha = 1$, if $\mu$ is Lebesgue measure, and if $\psi$ is nonincreasing, then the converse to Theorem~\ref{theoremweiss} holds.

Let us note that it is possible to have an absolutely $\alpha$-decaying measure $\mu$ and a nonincreasing function $\psi$ for which the converse to Theorem~\ref{theoremweiss} does not hold. Indeed, any absolutely $\alpha$-decaying measure is also absolutely $\alpha'$-decaying for any $0 < \alpha' < \alpha$, but it is easy to find a function $\psi$ such that the series (\ref{weiss}) converges with $\alpha = \alpha$ and diverges with $\alpha = \alpha'$. Thus, if we are to find a converse to Theorem~\ref{theoremweiss}, then $\mu$ being absolutely $\alpha$-decaying is not the right hypothesis.
\begin{question}[Weiss, private communication] \label{questionweiss}
What is the right hypothesis?
\end{question}

To clarify the question further, we give the following answer, which was also known to Weiss:
\begin{answer}
\label{answertoweiss}
The right hypothesis cannot be a geometrical hypothesis.
\end{answer}
By a ``geometrical'' hypothesis, we mean a hypothesis which is invariant under translations, i.e. if $\mu$ satisfies the hypothesis then any translated version of $\mu$ also satisfies the hypothesis. For example, the hypothesis that $\mu$ is absolutely $\alpha$-decaying and the hypothesis that $\mu$ is Ahlfors $\alpha$-regular are both geometrical hypotheses. Answer~\ref{answertoweiss} can be restated more formally as follows:
\begin{theorem}
\label{theoremanswer1}
Let $\nu$ be a measure on $\R$ and fix $0 < \alpha < 1$. Then there exists $y\in\R$ and a nonincreasing function $\psi$ such that if
\[
\mu = \nu\circ(x\mapsto x + y)^{-1}
\]
then the converse to Theorem~\ref{theoremweiss} does not hold.
\end{theorem}
\begin{proof}
Let $\psi$ be any function for which (\ref{weiss}) converges at $\alpha = 1$ but diverges at $\alpha = \alpha$. For example,
\[
\psi(q) = \frac{1}{q^2\log^{1/\alpha}(q)}\cdot
\]
By Khinchine's theorem, almost no point with respect to Lebesgue measure is $\psi$-approximable. By Fubini's theorem, for Lebesgue almost every $y\in\R$ and for $\nu$-almost every $x$, the point $x + y$ is $\psi$-approximable. This implies the existence of a number $y$ as stated in the theorem (in fact it implies the full measure of such $y$s).
\end{proof}
On the other hand, we give a positive answer to Question \ref{questionweiss} by exhibiting a large class of measures for which the converse to Theorem~\ref{theoremweiss} holds. To construct these, 
}

In the process of proving Theorems \ref{theoremconversestrong} and \ref{theoremexistence}, we establish the following criterion for determining whether $\HH^\alpha\given_{J_I}$ is Ahlfors regular. This improves more complicated criteria which can be found in \cite{MU1}.
\begin{reptheorem}{theoremcombinatorialcharacterization}[Abridged]
Fix an infinite set $I\subset\N$, and let $h = \HD(J_I)$. The following are equivalent:
\begin{itemize}
\item[(a)] $\HH^h(J_I) > 0$ and $\PP^h(J_I) < \infty$.
\item[(b1)] $\HH^h\given_{J_I}$ is Ahlfors $h$-regular.
\item[(c1)] For all $y\in I$ and $r\geq 1$
\[
\#(B(y,r)\cap I) \asymp r^h.
\]
\end{itemize}
\end{reptheorem}
Thus the Ahlfors regularity of $J_I$ is equivalent to the ``dual Ahlfors regularity'' of the generating set $I$.

Note that it is possible for (c1) to be satisfied for some $h\neq\HD(J_I)$. In such a case, the set $J_I$ is not Ahlfors regular.

The structure of the paper is as follows: In Section~\ref{sectiontheoremextremal}, we will prove Theorem~\ref{theoremextremal}. In Section~\ref{sectionIFS}, we will recall some basic definitions and theorems from the theory of conformal iterated function systems, which are needed to prove Theorems \ref{theoremcounterexample}, \ref{theoremcombinatorialcharacterization}, \ref{theoremconversestrong}, and \ref{theoremexistence}. In Section~\ref{sectionconformalextremal}, we will give some examples of measures which satisfy the hypotheses of Theorem~\ref{theoremextremal}, and we shall prove Theorem~\ref{theoremcounterexample}. In Section~\ref{sectioncombinatorialcharacterization} we will discuss various characterizations of Ahlfors regularity and semiregularity of $J_I$, and we shall prove Theorem~\ref{theoremcombinatorialcharacterization}. In Section~\ref{sectiontheoremconverse} we shall prove Theorem~\ref{theoremconversestrong} and in Section \ref{sectiontheoremexistence} we shall prove Theorem~\ref{theoremexistence}.

The interdependence of the sections is as follows: Section~\ref{sectionconformalextremal} depends on Sections \ref{sectiontheoremextremal} and \ref{sectionIFS}; Section~\ref{sectioncombinatorialcharacterization} depends on Section~\ref{sectionIFS}; Section~\ref{sectiontheoremconverse} depends on Sections \ref{sectiontheoremextremal}, \ref{sectionIFS}, and \ref{sectioncombinatorialcharacterization}; Section~\ref{sectiontheoremexistence} depends on \ref{sectionIFS} and \ref{sectioncombinatorialcharacterization}.

{\bf{Acknowledgments}}: The authors would like to thank both D. Y. Kleinbock and B. Weiss for reading the manuscript and making helpful suggestions. The third-named author was supported in part by the NSF Grant DMS 1001874.
% B. Weiss for suggesting Question~\ref{questionweiss}, and would like to thank 
\section{Proof of Theorem~\ref{theoremextremal}} \label{sectiontheoremextremal}
In this section we will prove the following theorem:

\begin{theorem}
\label{theoremextremal}
If $\mu$ is a probability measure on $\OI\butnot\Q$ invariant with respect to the Gauss map $G$ with finite Lyapunov exponent $\chi_\mu(G)$, then $\mu$ is extremal.
\end{theorem}

%\subsection{A relation between continued fractions and Diophantine approximation}
In the proof of Theorem~\ref{theoremextremal}, we will make use of a relation between the continued fraction expansion of an irrational $x\in\OI$ with its Diophantine properties, which we present as follows:

\begin{definition}
For a function $\psi:\N\rightarrow(0,\infty)$, let us say that $x$ is \emph{$\psi$-well approximable} if $x$ is $\varepsilon\psi$-approximable for every $\varepsilon > 0$.
\end{definition}
\begin{remark}
$\psi$-well approximability implies $\psi$-approximability but not vice-versa; for example, if $x$ is a badly approximable number and $\psi(q) = 1/q^2$ then $x$ is $\psi$-approximable but not $\psi$-well approximable.
\end{remark}

\begin{lemma}
\label{lemmaba}
Fix an irrational $x\in\OI$ and let $[0;\omega_0,\omega_1,\ldots]$ be the continued fraction expansion of $x$. Let $(p_n/q_n)_{n = 0}^\infty$ be the convergents of $x$. Let $\psi:\N\rightarrow(0,\infty)$ be a function satisfying
\[
\psi(q)\leq \frac{1}{q^2}
\]
for all $q$. Then $x$ is $\psi$-well approximable if and only if for every $K > 0$ there exist infinitely many $n\in\N$ such that
\begin{equation}
\label{badlyapproximable2}
\omega_n \geq K\phi(q_n),
\end{equation}
where
\[
\phi(q) := \frac{1}{q^2\psi(q)} \geq 1.
\]
\end{lemma}
\begin{remark}
It is possible to deduce Lemma \ref{lemmaba} from Theorem 8.5 of \cite{Dlecturenotes}, which is proved in a similar manner, but we include the proof for completeness. The ideas of this proof may also be found in the proof of Theorem 32 in \cite{Khi}.
\end{remark}
\begin{proof}[Proof of Lemma \ref{lemmaba}]
Fix an irrational $x\in\OI$. We recall the following well-known facts (see e.g. \cite{Khi} or \cite{Bugeaud}):
\begin{itemize}
\item[(i)] If $p/q$ is a rational approximation of $x$ such that $|x - p/q| < 1/(2q^2)$, then $p/q$ is a convergent of $x$.
\item[(ii)] For every $n\in\N$ we have
\begin{equation}
\label{xpqbounds}
\frac{1}{q_n(q_n + q_{n + 1})} < \left|x - \frac{p_n}{q_n}\right| < \frac{1}{q_n q_{n + 1}}
\end{equation}
and
\begin{equation}
\label{qnrecursion}
q_{n + 1} = \omega_n q_n + q_{n - 1}.
\end{equation}
\end{itemize}
From (i), it follows that for any $0 < \varepsilon\leq 1/2$, (\ref{approximable}) cannot be satisfied for any $p/q$ which is not a convergent. Thus, we may restrict our attention to approximations of $x$ which are convergents. Fix $n\in\N$, and note that by (\ref{qnrecursion}) we have
\[
q_{n + 1} \asymp \omega_n q_n, \footnote{Here and from now on $\asymp$ denotes a multiplicative asymptotic.}
\]
and thus (\ref{xpqbounds}) implies
\[
\left|x - \frac{p_n}{q_n}\right| \asymp \frac{1}{q_n^2 \omega_n}\cdot
\]
Thus, $x$ is $\psi$-well approximable if and only if for every $\varepsilon > 0$ there exist infinitely many $n\in\N$ such that
\[
\frac{1}{q_n^2 \omega_n} \geq \varepsilon\psi(q_n) = \frac{\varepsilon}{q_n^2\phi(q_n)}\cdot
\]
(We are using the ``$\varepsilon$'' to absorb the constant coming from the asymptotic.) Rearranging and letting $K = 1/\varepsilon$ yields the lemma.
\end{proof}

\ignore{

\begin{corollary}\label{lemmabacor1} 
Fix an irrational $x\in\OI$. Suppose that $\phi$ is \emph{very slowly varying},
i.e. there exists $K>0$ such that for all $t>0$, we have
$\phi(t^2)\leq K\phi(t)$. Then $x$ is badly approximable with respect
to $\phi$ if and 
only if there exists $\Delta < \infty$ such that for all $n\in\N$, 
\[
\eta(G^n(x))\le\Delta + \Phi\left(\sum_{j = 0}^{n - 1}\eta(G^j(x))\right),
\]
where $\Phi(t) = \log(\phi(e^t))$.
\end{corollary}
\noindent For example, for every $c > 0$ the function $\phi(t) = \log^c(t)$ is
very slowly varying.

}
\begin{definition}
\label{definitionxieta}
For $x\in\OI$, let
\[\xi(x) = \lfloor 1/x\rfloor\]
be the first entry in the continued fraction expansion of $x$, so that $\xi(G^n(x)) = \omega_n$ for all $n$. Let
\[\eta = \log(1 + \xi).\]
\end{definition}

\begin{corollary}
\label{corollaryvwa}
\noindent Fix an irrational $x\in\OI$ and let $[0;\omega_0,\omega_1,\ldots]$ be the continued fraction expansion of $x$. Then the following are equivalent:
\begin{itemize}
\item[(i)] The number $x$ is very well approximable. 
\item[(ii)] There exists $c > 0$ such that for infinitely many $n\in\N$,
\begin{equation}
\label{omegacomega}
\log(1 + \omega_n) \geq c\sum_{j = 0}^{n - 1}\log(1 + \omega_j),
\end{equation}
or equivalently,
\begin{equation}
\label{etaceta}
\eta(G^n(x)) \geq c\sum_{j = 0}^{n - 1}\eta(G^j(x)).
\end{equation}
\end{itemize}
\end{corollary}
Formula (\ref{etaceta}) will be more useful than (\ref{omegacomega}) for our ergodic theory purposes.
\begin{proof}
We first give some bounds for $q_n$ in terms of the continued fraction entries $\omega_0,\ldots,\omega_{n - 1}$. The upper bound is easy: the
recursion equation (\ref{qnrecursion}) implies that
\[
q_n \leq \prod_{j = 0}^{n - 1} (\omega_j + 1).
\]
In the other direction, we divide into cases according to whether $n$
is even or odd. If $n = 2k$ then 
\[
q_n \geq \prod_{j = 0}^{k - 1}(\omega_{2j} \omega_{2j + 1} + 1) \geq
\prod_{j = 0}^{n - 1} \sqrt{\omega_j + 1} 
\]
and if $n = 2k + 1$ then
\[
q_n \geq \omega_{n - 1} q_{n - 1} \geq \frac{1}{\sqrt 2}\prod_{j =
  0}^{n - 1} \sqrt{\omega_j + 1}. 
\]
Let $t\ge 0$. Taking logarithms, we can
rewrite the above inequalities as 
\begin{equation}
\label{qnbounds}
\frac{1}{2}\sum_{j = 0}^{n - 1} \eta(G^j(x)) - \log(\sqrt{2}) \leq
\log(q_n) \leq \sum_{j = 0}^{n - 1}\eta(G^j(x)).
\end{equation}
Now the corollary follows immediately from (\ref{qnbounds}) together with Lemma \ref{lemmaba}, and the following characterization of the set $\VWA$ of very well approximable numbers:
\begin{quote}
An irrational $x\in\OI$ is very well approximable if and only if there exists $c > 0$ such that it is $\psi$-well approximable where $\psi(q) = q^{-(2 + c)}$.
\end{quote}
\end{proof}

Using (\ref{qnbounds}) and Lemma~\ref{lemmaba}, we also deduce the following:
\begin{corollary}
\label{corollaryliouville}  
Fix an irrational $x\in\OI$. Then the following are equivalent:
\begin{itemize}
\item[(i)] $x$ is a Liouville number. 
\item[(ii)] For all $c > 0$, there exist infinitely many $n\in\N$ such that \textup{(\ref{etaceta})} holds.
\end{itemize}
\end{corollary}

\begin{remark}
\label{remarkba}
The well-known fact that an irrational $x\in\OI$ is badly approximable if and only if its continued fraction entries are bounded is also a corollary of Lemma~\ref{lemmaba}.
\end{remark}

\begin{proof}[Proof of Theorem~\ref{theoremextremal}]
We first of all note that it suffices to consider the case where $\mu$ is ergodic with respect to $G$, since if $\chi_\mu$ is finite, then $\chi_\nu$ is finite for almost all measures $\nu$ in the ergodic decomposition of $\mu$. Since the set $\VWA$ is invariant with respect to the Gauss map, it follows that any ergodic measure must give either zero or full measure to $\VWA$.

Let $\mu$ be an ergodic invariant measure whose Lyapunov exponent $\chi_\mu = \int\log|G'|\dee\mu$ is finite. Let $\eta$ be as in Definition~\ref{definitionxieta}. Since $\eta(x)\asymp -2\log(x) = \log|G'(x)|$, it follows that $\int\eta\dee\mu$ is also finite. On the other hand, $\eta$ is a strictly positive function and so we have
\begin{equation}\label{intetabounds}
0 < \int\eta\dee\mu < \infty. 
\end{equation}
We claim that $\mu$ is extremal. Suppose to the contrary that $\mu$-almost every number $x\in\OI$ is
very well approximable. It then follows from
Corollary~\ref{corollaryvwa}, \eqref{intetabounds}, and the Birkhoff
ergodic theorem that for $\mu$-almost all such numbers $x$, we have that
\[
\liminf_{n\to\infty}\frac1n\sum_{j= 0}^{n - 1}\eta(G^j(x))
\le \frac1{c_x}\limsup_{n\to\infty}\frac1n\eta(G^n(x))
=0
\]
with some $c_x>0$ coming from Corollary~\ref{corollaryvwa}.
Invoking the Birkhoff ergodic theorem again, we conclude that
$\int\eta\dee\mu\le 0$. This contradiction finishes the proof.
\end{proof}

\section{Iterated function systems and conformal measures} \label{sectionIFS}
Our main example of a measure invariant with respect to the Gauss map will be the unique invariant probability measure absolutely continuous to a conformal measure associated with an iterated function system consisting of inverse branches of the Gauss map. In this section we recall the definitions and main theorems. All theorems in this section except for those in Subsection~\ref{subsectiontwolemmas} were proven first in \cite{MU1} and also in a more general context in \cite{MU2}.

\subsection{IFSs and continued fractions}
For each $i\in\N$, consider the map $g_i:[0,1]\rightarrow[0,1]$ defined by
\[
g_i(x) = \frac{1}{i + x}\cdot
\]
The map $g_i$ is an inverse branch of the Gauss map $G$. For any set $I\subset\N$, the collection of maps $\SS_I = \{g_i\}_{i\in I}$ is a conformal iterated function system (see \cite{MU1} or \cite{MU2} for the definition).

Given
$\omega=\omega_0\omega_1\omega_2\ldots\omega_{n - 1}\in\N^n$, let 
\[
g_\omega
:=g_{\omega_0}\circ g_{\omega_1}\ldots\circ g_{\omega_{n - 1}}:\OI\to\OI,
\]
so that
\[
g_\omega(x)
=\cfrac1{\omega_0+\cfrac1{\omega_1+\cfrac1{
  \ddots+\cfrac1{\omega_{n - 1}+x}}}}\;.
\]
In particular,
\[
g_\omega(0) = [0;\omega_0,\omega_1,\ldots,\omega_{n - 1}].
\]
Let
\[
J_I = \bigcap_{n\in\N}\bigcup_{\omega\in I^\N}g_\omega([0,1]).
\]
The set $J_I$ is called the \emph{limit set} of the IFS $\SS_I$.
It coincides with the set of all irrational numbers in $\OI$ whose continued
fraction entries all lie in $I$. If $I$ is infinite, then the set $\cl{J_I}\butnot J_I$ consists of the set of all rational numbers whose continued fraction entries all lie in $I$. Moreover, $J_I$ is forward invariant under
the Gauss map $G$, i.e.
\[
G(J_I)=J_I.
\]

\subsection{A formula for the Hausdorff dimension of $J_I$}
Fix $I\subset\N$. A famous formula of R. Bowen relates the Hausdorff dimension of $J_I$ to an invariant of the IFS $\SS_I$ called the topological pressure. Below we describe this invariant and state Bowen's formula.

Given $t\ge 0$, the following limit exists and is called the \emph{topological pressure} of the IFS $\SS_I$ at the parameter $t$:
\[
P_I(t):=\lim_{n\to\infty}\frac1n\log\sum_{\omega\in I^n}\|g_\omega'\|_\infty^t.
\]

\begin{theorem}[Bowen's formula; Theorem 4.2.13 of \cite{MU2}]\label{BowenFormula}
For any set $I\subseteq\N$,
\[
\HD(J_I) = \inf\{t\ge 0:P_I(t)\le 0\}.
\]
In particular, $\HD(J_I)$ is the unique zero of $P_I$ if such a zero exists.
\end{theorem}

\noindent We will need also the following theorem:

\begin{theorem}[Theorem 2.1.5 of \cite{MU2}]\label{continuityofpressure}
Given $t\ge 0$, for each set $I\subseteq\N$,
\[
P_I(t)=\lim_{N\to\infty}P_{I\cap\{1,\ldots,N\}}(t).
\]
\end{theorem}

\subsection{Conformal measures}
Conformal measures are an important tool for understanding the geometry of the limit set $J_I$. In many cases they coincide with either the normalized Hausdorff measure or the normalized packing measure.
\begin{definition}
\label{definitionconformal}
Fix $t\ge 0$ and $I\subset\N$. A probability measure $m$ on $\OI$ is called \emph{$t$-conformal} with respect to the iterated function system $\SS_I$ if $m(J_I)=1$ and if
\[
m(g_i(A))=\int_A|g_i'|^t\dee m
\]
for every Borel set $A\subseteq \OI$ and for every $i\in I$.
\end{definition}
\begin{definition}
\label{definitionregular}
Fix $I\subset\N$. The system $\SS_I$ is said to be \emph{regular} if there exists $t\geq 0$ such that $P_I(t) = 0$.
\end{definition}
\begin{proposition}[Theorem 4.2.9 of \cite{MU2}]
\label{propositionregularequivalent}
Fix $I\subset\N$. The following are equivalent:
\begin{itemize}
\item[(a)] The IFS $\SS_I$ is regular.
\item[(b)] There exists a measure $m$ and $t\geq 0$ such that $m$ is $t$-conformal.
\end{itemize}
Furthermore, in this case $m$ and $t$ are both unique, and $P(t) = 0$.
\end{proposition}

\begin{corollary}\label{corollaryconformal}
Fix $t\geq 0$ and $I\subset\N$. Then if $P(t) = 0$, then there exists a measure $m$ which is $t$-conformal.
\end{corollary}
\begin{proof}
Since $P(t) = 0$, it follows that the IFS $\SS_I$ is regular, so by Proposition \ref{propositionregularequivalent}, there exists a measure $m$ and a number $t'\geq 0$ such that $m$ is $t'$-conformal and $P(t') = 0$. But since $P$ is strictly decreasing (part (b) of Proposition 4.2.8 of \cite{MU2}), we have $t = t'$, so $m$ is $t$-conformal.
\end{proof}

\begin{proposition}\label{propositioninvariantmeasure}
Fix $I\subset\N$, and suppose that the IFS $\SS_I$ is regular. Let $m_I$ be the unique conformal measure, and let $h = \HD(J_I)$, so that $m_I$ is $h$-conformal. Then there exists a unique Borel probability $G$-invariant measure $\mu_I$ on $J_I$ absolutely continuous with respect to $m_I$. This measure is ergodic and equivalent to $m_I$. The logarithm of the corresponding Radon-Nikodym derivative is a bounded function on $J_I$.
\end{proposition}
\begin{proof}
See Theorem 2.4.3 and part (c) of Corollary 2.7.5 in \cite{MU2}.
\end{proof}

\begin{proposition}
\label{propositionHausdorffpacking}~
\begin{itemize}
\item[(a)] If the $h_I$-dimensional Hausdorff measure of $J_I$ is positive (it is always finite), then the system $S_I$ is regular and we have
\[
m_I = \frac{\HH^{h_I}\given_{J_I}}{\HH^{h_I}(J_I)}\cdot
\]
\item[(b)] If the $h_I$-dimensional packing measure of $J_I$ is finite (it is always positive), then the system $S_I$ is regular and we have
\[
m_I = \frac{\PP^{h_I}\given_{J_I}}{\PP^{h_I}(J_I)}\cdot
\]
\end{itemize}
\end{proposition}
\begin{proof}
An argument analogous to the proof of the change of variables formula demonstrates that both of the above expressions are $h_I$-conformal. The proposition therefore follows from Proposition \ref{propositionregularequivalent}.
\end{proof}

\subsection{Regularity properties of the IFS $\SS_I$}
Fix $I\subset\N$. In this subsection we discuss properties of the IFS $\SS_I$ that are stronger than just regularity.

Let 
\[
\theta_I := \inf\{t\ge 0:P_I(t) < +\infty\}.
\]
We have the following simple characterization of the number $\theta_I$:

\begin{proposition}\label{theta}
\[
\theta_I
=\inf\{t\ge 0:\sum_{i\in I}\|g_i'\|_\infty^t<+\infty\}
=\inf\{t\ge 0:\sum_{i\in I}i^{-2t}<+\infty\}.
\]
\end{proposition}
\begin{proof}
The first equation is part (a) of Proposition 4.2.8 of \cite{MU2}; the second follows from the fact that $\|g_i'\|_\infty = i^{-2}$ for all $i\in\N$.
\end{proof}

\begin{definition}
Fix $I\subset\N$. The system $\SS_I$ is said to be \emph{strongly regular} if there exists $t\ge 0$ such that $0 < P_I(t) < +\infty$, and it is called \emph{cofinitely regular} (or \emph{hereditarily regular}) if $P_I(\theta_I) = +\infty$.
\end{definition}
We have the following:

\begin{proposition}[Theorem 4.3.5 of \cite{MU2}]\label{propositionregularity}~
\begin{itemize}
\item[(a)] Every cofinitely regular system is strongly regular and every strongly regular system is regular.
\item[(b)] For each strongly regular system $\SS_I$, we have $\HD(J_I) > \theta_I$.
\end{itemize}
\end{proposition}

\begin{proposition}[Theorem 4.3.4 of \cite{MU2}]\label{propositioncofinitelyregular}
Fix $I\subset\N$. The system $\SS_I$ is cofinitely regular if and only if the series
\[
\sum_{i\in I}\|g_i'\|_\infty^{\theta_I} = \sum_{i\in I}i^{-2\theta_I}
\]
diverges.
\end{proposition}

Recall from Section~\ref{sectionintroduction} that if $\mu$ is a probability measure on $\OI$ is invariant with respect to the Gauss map $G:\OI\rightarrow\OI$, then the integral
\[
\chi_\mu(G)=\int\log|G'|\dee\mu
\]
is called the \emph{Lyapunov exponent} of the measure $\mu$ with respect to the Gauss map $G$. We have the following:

\begin{proposition}\label{srimpfle}
Fix $I\subseteq\N$. If the system $\SS_I$ is strongly regular, then $\chi_{\mu_I}(G) < +\infty$.
\end{proposition}
\begin{proof}
By part (b) of Proposition~\ref{propositionregularity}, we have $h_I > \theta_I$. Fix $\theta_I < t < h_I$. Then by the definition of $\theta_I$, the series
\[
\sum_{i\in I}i^{-2t}
\]
converges. Now we can estimate the Lyapunov exponent of $\mu_I$ as follows:
\begin{align*}
\int\log|G'|\dee\mu_I \asymp \int\eta\dee m_I
&= \sum_{i\in I}\log(1 + i)m_I(g_i(\OI))\\
&\asymp \sum_{i\in I}\log(1 + i)i^{-2h_I}
\lesssim \sum_{i\in I}i^{-2t} < \infty.
\end{align*}
\end{proof}

\subsection{Two lemmas} \label{subsectiontwolemmas}
Each of the two lemmas in this subsection will be used several times throughout the remainder of the paper.

For every set $I\subseteq \N$ and for every $t\ge 0$ let
\[
\lambda_t(I) = e^{P_I(t)}.
\]

\begin{lemma}[Lemma 4.3 of \cite{KZ}]
\label{lemmaKZ}
Fix $\delta > 0$. Let $i\geq 2$ and let $I$ be a finite subset of $\N\butnot\{i\}$. Then
\begin{equation}
\label{KZ}
\lambda_\delta(I) + \left(\frac{1}{i + 1}\right)^{2\delta}
\leq \lambda_\delta(I\cup\{i\})
\leq \lambda_\delta(I) + \left(\frac{2}{i + 2}\right)^{2\delta}\cdot
\end{equation}
\end{lemma}
\begin{remark}
Applying Theorem~\ref{continuityofpressure} to this lemma, we conclude
that (\ref{KZ}) holds in fact for all sets $I\subset \N\butnot\{i\}$. 
\end{remark}

Recall that we have defined $\xi(x) = \lfloor 1/x\rfloor$ to be the first entry of the continued fraction expansion of $x$. For any $\omega\in \N^n$, let
\[
S_\omega := g_\omega(\OI) = \{x\in\OI:\xi(G^j(x)) = \omega_j\all j = 0,\ldots,n - 1\},
\]
i.e. $S_\omega$ is the set of all numbers whose continued fraction expansions begin with the sequence $\omega_0,\ldots,\omega_{n - 1}$. Furthermore, for each $k\in\N$ let
\[
S_{\omega,k}^+=\bigcup_{i\leq k}S_{\omega i}.
\]

\begin{lemma}
\label{lemmaSomega}
Fix $I\subseteq\N$, and suppose that the IFS $\SS_I$ is regular. Let $h_I$ be the Hausdorff dimension of $J_I$ and let $m_I$ be the unique $h_I$-conformal measure of $\SS_I$. Then
\begin{equation}
\label{SomegaiSomega}
\frac{m_I(S_{\omega i})}{m_I(S_\omega)}
\geq \frac{1}{4^{h_I}}i^{-2h_I}
\end{equation}
and
\begin{equation}\label{220120504}
\frac{m_I( S_{\omega,k}^+)}{m_I(S_\omega)} 
\leq 1 - \frac{1}{4^{h_I}}\sum_{\substack{i \in I \\ i > k}}i^{-2h_I}.
\end{equation}
\end{lemma}
\begin{proof}
It is clear that (\ref{220120504}) follows from (\ref{SomegaiSomega}). To demonstrate (\ref{SomegaiSomega}), note that since $m_I$ is $h_I$-conformal we have
\begin{align*}
\frac{m_I(S_{\omega i})}{m_I(S_\omega)}
= \frac{m_I(g_{\omega i}(\OI))}{m_I(g_\omega(\OI))}
&= \frac{\int |(g_{\omega i}'(x)|^{h_I}\dee m_I(x)}{\int |g_\omega'(x)|^{h_I}\dee m_I(x)}\\
&\geq \frac{\min_{\OI}|g_\omega'|^{h_I}}{\max_{\OI}|g_\omega'|^{h_I}}\min_{\OI}|g_i'|^{h_I}.
\end{align*}
On the other hand, we have (see \cite{MU1}, line -10 of p.4997, or by direct computation)
\[
\frac{\max_{\OI}|g_\omega'|}{\min_{\OI}|g_\omega'|} \leq 4
\]
which yields (\ref{SomegaiSomega}).
\end{proof}

\section{Extremality of conformal measures} \label{sectionconformalextremal}
Fix $I\subset\N$, and suppose that the IFS $\SS_I$ is regular. In this section we discuss the extremality of the measures $m_I$ and $\mu_I$ defined in Section~\ref{sectionIFS}. Note that since $m_I$ and $\mu_I$ are absolutely continuous to each other, $m_I$ is extremal if and only if $\mu_I$ is.

By Theorem~\ref{theoremextremal}, if $\chi_{\mu_I} < + \infty$, then $\mu_I$ is extremal. By Proposition~\ref{srimpfle}, if $\SS_I$ is strongly regular, then $\chi_{\mu_I} < \infty$. The following proposition gives very general sufficient conditions for $\SS_I$ to be strongly regular:

\begin{proposition}\label{scfsr}
If $I\subseteq\N$, then any of the following three conditions entail strong regularity of the iterated function system $\SS_I$, and thus the extremality of the measures $m_I$ and $\mu_I$:
\begin{itemize}
\item[(a)] $I$ is finite.
\item[(b)] The series $\sum_{a\in I}a^{-2\theta_I}$ diverges.
\item[(c)] The Hausdorff dimension of the limit set of the IFS is strictly
  greater than $1/2$. 
\item[(d)] $1,2\in I$.
\end{itemize}
\end{proposition}
\begin{proof}
Item (a) follows directly from the definition (we have $0 < P_I(0) < \infty$); item (b) follows from (a) of Proposition~\ref{propositionregularity} and Proposition~\ref{propositioncofinitelyregular}; item (c) follows from Theorem 4.3.10 of \cite{MU2} along with the observation that $\theta_I\le\theta_\N=1/2$; item (d) follows from item (c) and the fact, proven in \cite{Go}, that $h_{\{1,2\}} = \HD(J_{\{1,2\}}) > 1/2$.
\end{proof}
\begin{remark}
In case (a), the extremality of $\mu_I$ is obvious since its topological support $J_I$ is contained in the set of badly approximable numbers (see Remark~\ref{remarkba}).
\end{remark}
\begin{remark}
The main result of \cite{Ur}, namely, the extremality part of Theorem 4.5 of that paper, can be deduced from part (b) of Proposition \ref{scfsr}.
\end{remark}
\begin{example}
Fix $a\geq 2$ and let $I$ be the geometric series $I = \{a,a^2,\ldots\}$. Then condition (b) of Proposition~\ref{scfsr} is satisfied. Thus the measure $\mu_I$ is extremal. On the other hand, $\mu_I$ is not absolutely decaying (see below), so the extremality of $\mu_I$ does not follow from Weiss's theorem \cite{Weiss}.
\end{example}
\begin{proof}[Proof that $\mu_I$ is not absolutely decaying]
Fix $n\in\N$, and let $x_n = a^{-n}\in \cl{J_I}$. Then
\[
B\left(x_n,\frac{1}{a^n} - \frac{1}{a^n + 1}\right)\cap J_I = B\left(x_n,\frac{1}{a^n} - \frac{1}{a^{n + 1}}\right)\cap J_I.
\]
If $\mu$ were absolutely $\alpha$-decaying, we would therefore have
\begin{align*}
1 = \frac{\mu\left[B\left(x_n,\frac{1}{a^n} - \frac{1}{a^n + 1}\right)\right]}{\mu\left[B\left(x_n,\frac{1}{a^n} - \frac{1}{a^{n + 1}}\right)\right]}
&\leq C \left(\frac{\frac{1}{a^n} - \frac{1}{a^n + 1}}{\frac{1}{a^n} - \frac{1}{a^{n + 1}}}\right)^\alpha\\
&\asymp \left(\frac{1/a^{2n}}{1/a^n}\right)^\alpha = \frac{1}{a^{n\alpha}},
\end{align*}
which is a contradiction for $n$ large enough.
\end{proof}

The remainder of this section will be devoted to proving the following theorem:
\begin{theorem}
\label{theoremcounterexample}
There exists a measure $\mu$ invariant with respect to the Gauss map which gives full measure to the Liouville numbers. In particular, $\mu$ is not extremal.
\end{theorem}
The measure $\mu$ will be of the form $\mu_I$ for some $I\subset\N$ defining a regular system $\SS_I$.

Fix $0 < \delta \leq 1/2$, and define a sequence of finite subsets $I_N\subset\N$ recursively in the
following manner:
\begin{itemize}
\item[1.] Let $I_0 = \emptyset$.
\item[2.] Suppose that the set $I_{N - 1}$ has been defined. Let $M_{N - 1} = \max(I_{N - 1})$. (By convention let $\max(\emptyset) = 0$.)
\item[3.] Choose $m_N\in\N$ large enough so that:
\begin{align*}
\log(1 + m_N) &\geq N 4^N \log(1 + M_{N - 1})\\
\left(\frac{2}{m_N + 2}\right)^{2\delta} &\leq 2^{-N}.
\end{align*}
\item[4.] Let $R_N\subset\{m_N,\ldots\}$ be a finite set satisfying:
\begin{equation}
\label{inductionhypothesis}
1 - 2^{-(N - 1)} \leq \lambda_\delta(I_{N - 1}\cup R_N) < 1 - 2^{-N}.
\end{equation}
(The existence of such a set $R_N$ is verified below.)
\item[5.] Let $I_N = I_{N - 1}\cup R_N$ and then go back to step 2.
\end{itemize}
We now check that in step 4, it is always possible to find a set $R_N$ which satisfies (\ref{inductionhypothesis}). We first claim that
\begin{equation}
\label{lambdadeltamN}
\lambda_\delta(I_{N - 1}) < 1 - 2^{-N} < \lambda_\delta(I_{N - 1}\cup \{m_N,\ldots\}).
\end{equation}
Indeed, the left inequality follows from the induction hypothesis (or by direct computation in the case $N = 1$). The right hand side follows from Lemma~\ref{lemmaKZ} and the fact that the series
\[
\sum_{i = m_N}^\infty \left(\frac{1}{i + 1}\right)^{2\delta}
\]
diverges (since $\delta\leq 1/2$).

It follows from (\ref{lambdadeltamN}) that there exists $K\in\{m_N,\ldots\}$ so that
\begin{align*}
\lambda_\delta(I_{N - 1}\cup \{m_N,\ldots,K\}) &< 1 - 2^{-N} \\&\leq \lambda_\delta(I_{N - 1}\cup \{m_N,\ldots,K + 1\}).
\end{align*}
Let $R_N = \{m_N,\ldots,K\}$. By Lemma~\ref{lemmaKZ}, we have
\begin{align*}
\lambda_\delta(I_{N - 1}\cup &\{m_N,\ldots,K\})\geq\\
&\geq \lambda_\delta(I_{N - 1}\cup \{m_N,\ldots,K + 1\}) -
\left(\frac{2}{(K + 1) + 2}\right)^{2\delta}\\ 
&\geq 1 - 2^{-N} - 2^{-N}
= 1 - 2^{-(N - 1)}
\end{align*}
which demonstrates (\ref{inductionhypothesis}).

Let
\[
I = \bigcup_N I_N.
\]
By Theorem~\ref{continuityofpressure}, we have $\lambda_\delta(I) = 1$, and thus $P_I(\delta) = 0$. By Corollary~\ref{corollaryconformal} and Proposition~\ref{propositioninvariantmeasure} we have that $\HD(J_I) = \delta$, and that there exists a $\delta$-conformal measure $m_I$ and an absolutely continuous $G$-invariant measure $\mu_I$.
To complete the proof we need to show that $m_I$, and thus $\mu_I$, gives full measure to the set of Liouville numbers.
To this end, fix $N\in\N$. By Lemma~\ref{lemmaKZ} we have
\begin{equation}\label{120120505}
1 - \lambda_\delta(I\cap\{1,\ldots,M_N\})
\leq \sum_{\substack{i\in I \\ i > M_N}}\left(\frac{2}{2 + i}\right)^{2\delta}
\leq 4^\delta\sum_{\substack{i\in I \\ i > M_N}}\left(\frac{1}{1 + i}\right)^{2\delta}\cdot
\end{equation}
Fix $\omega = (\omega_j)_{j = 0}^{n - 1} \in \N^n$. It then follows from
\eqref{220120504} and \eqref{120120505} that
\[
\frac{m_I(S_{\omega,M_N}^+)}{m_I(S_\omega)} \leq 1 - \frac{1}{16^\delta}(1 - \lambda_\delta(I\cap\{1,\ldots,M_N\})),
\]
where $S_\omega$ and $S_{\omega,k}^+$ are defined as in Lemma~\ref{lemmaSomega}. Invoking (\ref{inductionhypothesis}) gives
\begin{equation}\label{220120505}
\frac{m_I(S_{\omega,M_N}^+)}{m_I(S_{\omega})} \leq 1 - c 2^{-N},
\end{equation}
where $c = 1/16^\delta$. Now for each $n\in\N$ let
\[
S_{n,N} = \{x\in\OI:\xi(G^j(x)) \leq M_N \all j = 0,\ldots,n - 1\}.
\]
Formula \eqref{220120505} yields
\[
\frac{m_I(S_{n + 1,N})}{m_I(S_{n,N})} \leq 1 - c 2^{-N}.
\]
Iterating yields
\[
m_I(S_{n,N}) \leq (1 - c 2^{-N})^n.
\]
Letting $n = 4^N$, we see that
\[
m_I(S_{4^N,N}) \leq e^{-c 2^N}
\]
and thus
\[
\sum_{N = 0}^\infty m_I(S_{4^N,N}) < \infty.
\]
Thus by the Borel-Cantelli lemma, $m_I$-almost every point $x\in J_I$ lies in only finitely many
sets of the form $S_{4^N,N}$. Fix such a point $x$, and we will show
that $x$ is a Liouville number. By Corollary~\ref{corollaryliouville}, it suffices to demonstrate that for all $c > 0$ we have 
\begin{equation}
\label{infinitelymany}
\eta(G^n(x)) \geq c \sum_{j = 0}^{n - 1}\eta(G^j(x))
\end{equation}
for infinitely many $n\in\N$, where $\eta$ is defined as in Definition~\ref{definitionxieta}. Indeed, for all but finitely many
$N\in\N$, we have $x\notin S_{4^N,N}$ and so there exists $n\leq 4^N$
such that $\xi(G^n(x)) > M_N$. Without loss of generality, we may
assume that $n$ is minimal with this property, i.e. $\xi(G^j(x)) \leq
M_N$ for all $j < n$. Now, since $I$ does not contain any numbers
between $M_N$ and $m_{N + 1}$, we have $\xi(G^n(x)) \geq m_{N + 1}$
and thus 
\begin{align*}
\eta(G^n(x)) &\geq \log(1 + m_{N + 1}) = N 4^N \log(1 + M_N)\\
&\geq N\sum_{j = 0}^{n - 1}\log(1 + \xi(G^j(x)))\\
&= N\sum_{j = 0}^{n - 1}\eta(G^j(x))
\end{align*}
which demonstrates that (\ref{infinitelymany}) has infinitely many solutions. Thus the proof of Theorem~\ref{theoremcounterexample} is complete.

\section{Combinatorial characterizations of Ahlfors regularity}\label{sectioncombinatorialcharacterization}
In this section we prove Theorem \ref{theoremcombinatorialcharacterization} which gives a combinatorial characterization of Ahlfors regularity of $J_I$. We begin by recalling the following theorems:

\begin{theorem}[Theorem 4.1 of \cite{MU1}]
\label{theoremlowerregular}
Fix a set $I\subset\N$, and suppose that the IFS $\SS_I$ is regular. Let $h = \HD(J_I)$, and let $m_I$ be an $h$-conformal measure. Then the following are equivalent:
\begin{itemize}
\item[(a)] $\HH^h(J_I) > 0$.
\item[(b)]~
\begin{equation}
\label{lowerregular1}
\sup_{k_1 < k_2}\frac{(k_1 k_2)^h}{(k_2 - k_1)^h}\sum_{\substack{i\in I \\ k_1\leq i\leq k_2}} i^{-2h} < \infty.
\end{equation}
\item[(c)] $m_I$ is Ahlfors $h$-lower regular i.e.
\[
m_I(B(x,r)) \lesssim r^h \all x\in J_I \all r\leq 1.
\]
\end{itemize}
\end{theorem}
\begin{proof}
The equivalence of (a) and (b) is proven in Theorem 4.1 of \cite{MU1}. The implication (a)$\Rightarrow$(c) follows from the last line of the proof of the implication (c)$\Rightarrow$(a) of Theorem 4.5.3 of \cite{MU2} (just before the mass distribution principle is applied), and the implication (c)$\Rightarrow$(a) is the mass distrubution principle.
\end{proof}

\begin{theorem}[Theorem 5.1 of \cite{MU1}]
\label{theoremupperregular}
Fix an infinite set $I\subset\N$, and suppose that the IFS $\SS_I$ is regular. Let $h = \HD(J_I)$, and let $m_I$ be an $h$-conformal measure. Then the following are equivalent:
\begin{itemize}
\item[(a)] $\PP^h(J_I) < \infty$.
\item[(b)] Both of the following hold:
\begin{align*}
\inf_{\substack{k_1 < k_2 \\ B\left(\frac{2 k_1 k_2}{k_1 + k_2},1\right)\cap I\neq\emptyset}}\frac{(k_1 k_2)^h}{(k_2 - k_1)^h}\sum_{\substack{i\in I \\ k_1\leq i\leq k_2}} i^{-2h} &> 0\\
\inf_{k\geq 1}k^h \sum_{\substack{i\in I \\ i\geq k}}i^{-2h} &> 0.
\end{align*}
\item[(c)] $m_I$ is Ahlfors $h$-upper regular i.e.
\[
m_I(B(x,r)) \gtrsim r^h \all x\in J_I \all r\leq 1.
\]
\end{itemize}
\end{theorem}
\noindent Note that the assumption that $I$ is infinite is necessary in this theorem since any finite IFS satisfies (a) and (c) but not (b).
\begin{proof}
The equivalence of (a) and (b) is proven in Theorem 5.1 of \cite{MU1}. The implication (a)$\Rightarrow$(c) follows from the last line of the proof of the implication (c)$\Rightarrow$(a) of Theorem 4.5.5 of \cite{MU2} (just before the mass distribution principle for packing measure is applied), and the implication (c)$\Rightarrow$(a) is the mass distrubution principle for packing measure.
\end{proof}

We can add new equivalences to Theorems \ref{theoremlowerregular} and \ref{theoremupperregular} as follows:
\begin{theorem}
\label{theoremlowerregularnew}
\textup{(a)-(c)} of Theorem \ref{theoremlowerregular} are equivalent to the following:
\begin{itemize}
\item[(d)] Both \textup{(i)} and \textup{(ii)} hold:
\begin{itemize}
\item[(i)] For all $y\in\N$ and for all $1\leq r\leq y/2$,
\[
\#(B(y,r)\cap I) \lesssim r^h.
\]
\item[(ii)] For all $k\in\N$,
\[
\sum_{\substack{i\in I\\ i > k}}i^{-2h} \lesssim k^{-h}.
\]
\end{itemize}
\end{itemize}
\end{theorem}
\begin{theorem}
\label{theoremupperregularnew}
\textup{(a)-(c)} of Theorem \ref{theoremupperregular} are equivalent to the following:
\begin{itemize}
\item[(d)] Both \textup{(i)} and \textup{(ii)} hold:
\begin{itemize}
\item[(i)] For all $y\in I$ and for all $1\leq r\leq y/2$,
\[
\#(B(y,r)\cap I) \gtrsim r^h.
\]
\item[(ii)] For all $k\in\N$,
\[
\sum_{\substack{i\in I\\ i > k}}i^{-2h} \gtrsim k^{-h}.
\]
\end{itemize}
\end{itemize}
\end{theorem}

\begin{proof}[Proof of Theorems \ref{theoremlowerregularnew} and \ref{theoremupperregularnew}]
By way of illustration we shall show that (d) of Theorem \ref{theoremlowerregularnew} implies (b) of Theorem \ref{theoremlowerregular}. The proof of the other implications are left to the reader.

Fix $k_1 < k_2$. If $I\cap[k_1,k_2] = \emptyset$, then the pair $(k_1,k_2)$ does not contribute to the supremum (\ref{lowerregular1}). Thus, suppose that $I\cap[k_1,k_2]\neq\emptyset$, and fix $y\in I\cap[k_1,k_2]$. Let $r = \max(k_2 - y,y - k_1)$. If $r\leq y/2$, then we have
\begin{align*}
\frac{2}{3}k_2 &\leq y \leq 2k_1\\
r &\leq k_2 - k_1 \leq 2r
\end{align*}
and thus by (d)(i) of Theorem \ref{theoremlowerregularnew} we have
\begin{align*}
\frac{(k_1 k_2)^h}{(k_2 - k_1)^h}\sum_{\substack{i\in I \\ k_1\leq i\leq k_2}} i^{-2h}
&\asymp \frac{(y^2)^h}{r^h}\sum_{\substack{i\in I \\ k_1\leq i\leq k_2}}y^{-2h}\\
&= r^{-h}\#(I\cap[k_1,k_2])\\
&\leq r^{-h}\#(I\cap B(y,r))
\lesssim r^{-h} r^h = 1.
\end{align*}
On the other hand, suppose that $r\geq y/2$. Then
\begin{align*}
k_2 - k_1 &\geq r\geq \frac{k_1}{2}\\
k_2 &\geq \frac{3}{2}k_1\\
k_2 - k_1 &\geq \frac{k_2}{3}
\end{align*}
and thus by (d)(ii) of Theorem \ref{theoremlowerregularnew} we have
\begin{align*}
\frac{(k_1 k_2)^h}{(k_2 - k_1)^h}\sum_{\substack{i\in I \\ k_1\leq i\leq k_2}} i^{-2h}
&\leq \frac{(k_1 k_2)^h}{(k_2/3)^h}\sum_{\substack{i\in I \\ i\geq k_1}} i^{-2h}\\
&\asymp k_1^h \sum_{\substack{i\in I \\ i\geq k_1}} i^{-2h}
\lesssim k_1^h k_1^{-h} = 1
\end{align*}
Thus either way we have
\[
\frac{(k_1 k_2)^h}{(k_2 - k_1)^h}\sum_{\substack{i\in I \\ k_1\leq i\leq k_2}} i^{-2h} \lesssim 1,
\]
which is equivalent to (\ref{lowerregular1}).
\end{proof}

If we restrict our attention to sets $I$ which satisfy both the conditions of Theorem~\ref{theoremlowerregular} and those of Theorem~\ref{theoremupperregular}, then we get even more characterizations:
\begin{theorem}
\label{theoremcombinatorialcharacterization}
Fix an infinite set $I\subset\N$, and let $h = h_I = \HD(J_I)$. Then \text{(a)-(c3)} are equivalent and imply \text{(d)-(e)}:
\begin{itemize}
\item[(a)] $\HH^h(J_I) > 0$ and $\PP^h(J_I) < \infty$.
\item[(b1)] $\HH^h\given_{J_I}$ is Ahlfors $h$-regular.
\item[(b2)] $\PP^h\given_{J_I}$ is Ahlfors $h$-regular.
\item[(b3)] The IFS $\SS_I$ is regular and $m_I$ is Ahlfors $h$-regular.
\item[(b4)] The IFS $\SS_I$ is regular and $\mu_I$ is Ahlfors $h$-regular.
\item[(c1)] For all $y\in I$ and $r\geq 1$
\begin{equation}
\label{combinatoriallyregular}
\#(B(y,r)\cap I) \asymp r^h.
\end{equation}
\item[(c2)] Both of the following hold:
\begin{itemize}
\item[(i)] \textup{(\ref{combinatoriallyregular})} holds for all $y\in I$ and $1\leq r\leq y/2$.
\item[(ii)] There exists $m\in\N$ such that for all $k\in\N$, we have $[k,mk]\cap I\neq\emptyset$.
\end{itemize}
\item[(c3)] Both of the following hold:
\begin{itemize}
\item[(i)] \textup{(\ref{combinatoriallyregular})} holds for all $y\in I$ and $1\leq r\leq y/2$.
\item[(ii)] For all $k\in\N$ we have
\begin{equation}
\label{hItail}
\sum_{\substack{i\in I\\ i > k}}i^{-2h} \asymp k^{-h}.
\end{equation}
\end{itemize}
\item[(d)] $\theta_I = h/2$.
\item[(e)] The IFS $\SS_I$ is cofinitely regular.
\end{itemize}
\end{theorem}
\begin{proof}
Let us first assume that $\SS_I$ is regular. Then the equivalence of (a), (b3), and (c3) follows directly from Theorems \ref{theoremlowerregularnew} and \ref{theoremupperregularnew}. The equivalence of (b3) and (b4) follows from Proposition~\ref{propositioninvariantmeasure}. To see that (b1) and (b3) are equivalent, note that by part (a) of Proposition~\ref{propositionHausdorffpacking}, if the equivalence fails then $\HH^h(J_I) = 0$. But in this case, clearly (b1) and (a) are both false, so since (a) is equivalent to (b3) we have (b1)$\Leftrightarrow$(b3). A similar argument yields the equivalence of (b2) and (b3).

We next show that (c1)$\Leftrightarrow$(c2)$\Leftrightarrow$(c3)$\Rightarrow$(d), (e). In these proofs we do not assume regularity of $\SS_I$.

\begin{proof}[Proof of \textup{(c3)$\Rightarrow$(c2)}] Suppose that (c3) holds. Let $C$ be the implied constant of (\ref{hItail}), and let $m = \lceil C^{2/h}\rceil + 1$. Then for any $k\in\N$, we have
\[
\sum_{\substack{i\in I\\ i > mk}}i^{-2h} \leq C (mk)^{-h}
< C^{-1} k^{-h} \leq \sum_{\substack{i\in I\\ i > k}}i^{-2h}
\]
which demonstrates that $[k,mk]\cap I\neq\emptyset$.
\QEDmod\end{proof}
\begin{proof}[Proof of \textup{(c2)$\Rightarrow$(c1), (c3), (d), (e)}] Suppose that (c2) holds. We claim that
\begin{equation}
\label{Ik3mk}
\#(I\cap[k,3mk]) \asymp k^h
\end{equation}
for all $k\in\N$. Indeed, the upper bound can be achieved by covering $I\cap [k,3mk]$ by finitely many sets of the form $B(y,y/2)$, where $y\in I\cap[k,3mk]$, and applying (\ref{combinatoriallyregular}). The lower bound follows from choosing a point $y\in I\cap[2k,2mk]$ and applying (\ref{combinatoriallyregular}) to the set $B(y,y/2)$.

From (\ref{Ik3mk}), we calculate that for any $t\geq 0$ and $k\in\N$ we have
\begin{align*}
\sum_{\substack{i\in I\\ i > k}}i^{-2t} &\asymp \sum_{n\in\N}\sum_{\substack{i\in I\\ (3m)^n k < i \leq (3m)^{n + 1} k}}i^{-2t}\\
&\asymp \sum_{n\in\N}[(3m)^n k]^h[(3m)^n k]^{-2t}
\end{align*}
which diverges if $t\leq h/2$ and is otherwise asymptotic to
\[
k^{h - 2t} < \infty.
\]
Specializing to the case $t = h$ yields (c3). Applying Proposition~\ref{theta} yields (d). Finally, Proposition~\ref{propositioninvariantmeasure} yields (e).

To demonstrate (c1), fix $y\in I$ and $r\geq 1$. If $r\leq y/2$, then we have (\ref{combinatoriallyregular}) for free. Thus, suppose $r > y/2$. Let $N = \lceil\log_{3m}(y + r)\rceil$. Then
\[
B(y,r)\subset \bigcup_{n = 0}^N \left[(3m)^n,(3m)^{n + 1}\right].
\]
On the other hand, for each $n\leq N$ we have from (\ref{Ik3mk})
\[
\#\left(I\cap\left[(3m)^n,(3m)^{n + 1}\right]\right) \asymp [(3m)^n]^h
\]
and summing yields
\[
\#(B(y,r)\cap I) \lesssim [(3m)^N]^h \asymp r^h.
\]
To get the lower bound, note that
\[\#(B(y,r)\cap I)\geq \#(B(y,y/2)\cap I)\asymp (y/2)^h\asymp y^h.\]
This bound is good enough unless $r\geq y$. In this case, let $k = \lfloor r/(3m)\rfloor$, and (\ref{Ik3mk}) yields the bound.
\QEDmod\end{proof}
\begin{proof}[Proof of \textup{(c1)$\Rightarrow$(c2)}] Similar to the proof of (c3)$\Rightarrow$(c2).
\QEDmod\end{proof}
This completes the proof of the theorem in the case where $\SS_I$ is regular. Suppose on the other hand that $\SS_I$ is not regular. Then (b3) and (b4) are clearly false. Applying parts (a) and (b) of Proposition~\ref{propositionHausdorffpacking} yields that (a), (b1), and (b2) are false. Applying part (a) of Proposition~\ref{propositionregularity} yields that (d) is false. Since we have (c1)$\Leftrightarrow$(c2)$\Leftrightarrow$(c3)$\Rightarrow$(d), and the proof of this did not depend on the regularity of $\SS_I$, we have that (c1)-(c3) are false. This yields the theorem.
\end{proof}

\section{Proof of Theorem~\ref{theoremconversestrong}}\label{sectiontheoremconverse}
In this section we will prove the following theorem:
\begin{theorem}
\label{theoremconversestrong}
Fix an infinite set $I\subset\N$, and let $h$ be the Hausdorff dimension of $J_I$. Assume that the $h$-dimensional Hausdorff measure $\HH^h$ restricted to $J_I$ is Ahlfors $h$-regular. Let $\mu = \HH^h\given_{J_I}$, and let $\psi:\N\rightarrow(0,+\infty)$ be a function such that the function $q\mapsto q^2\psi(q)$ is nonincreasing. Then $\mu$-almost every (resp. $\mu$-almost no) point is $\psi$-approximable, assuming that the series
\begin{equation}
\label{weiss}
\sum_{q = 1}^\infty q^{2\alpha - 1}\psi(q)^\alpha
\end{equation}
diverges (resp. converges).
\end{theorem}
As noted in the introduction, the convergence case follows from Weiss's theorem \cite{Weiss}.

Fix a function $\psi:\N\rightarrow(0,\infty)$, and suppose that the series (\ref{weiss}) diverges and that the function $q\mapsto q^2\psi(q)$ is nonincreasing. By (a) of Proposition~\ref{propositionHausdorffpacking}, we have $m_I\asymp \HH^h\given_{J_I}$, so to prove the theorem it suffices to show that $m_I$-almost every number is $\psi$-approximable. In fact, we will demonstrate the (slightly) stronger statement that $m_I$-almost every number is $\psi$-well approximable.

By (b1)$\Rightarrow$(e) of Theorem~\ref{theoremcombinatorialcharacterization}, the iterated function system $\SS_I = \{g_a\}_{a\in I}$ is cofinitely regular. Thus the Lyapunov exponent of $\mu_I$ is finite (Proposition~\ref{srimpfle} and part (a) of Proposition~\ref{propositionregularity}) and in particular $0 < \int\eta\dee\mu_I < \infty$ (see (\ref{intetabounds})), where $\eta$ is defined as in Definition~\ref{definitionxieta}. Thus by the Birkhoff ergodic theorem, we have
\[
\frac{1}{n}\sum_{j = 0}^{n - 1}\eta(G^j(x)) \tendsto n E := \int\eta\dee\mu_I
\]
for $\mu_I$-almost every $x\in\OI$. Combining the above equation with (\ref{qnbounds}) gives
\begin{equation}
\label{birkhoffnew}
\frac{E}{2} \leq \liminf_{n\rightarrow\infty}\frac{1}{n}\log(q_n) \leq \limsup_{n\rightarrow\infty}\frac{1}{n}\log(q_n) \leq E.
\end{equation}

Let $x\in\OI$ be a point such that (\ref{birkhoffnew}) holds but which is not $\psi$-well approximable. By Lemma~\ref{lemmaba} there exists $K > 0$ such that for all $n\in\N$, (\ref{badlyapproximable2}) fails to hold. Combining (\ref{birkhoffnew}), the negation of (\ref{badlyapproximable2}), and the fact that $\phi(q) = \frac{1}{q^2\psi(q)}$ is nondecreasing yields
\begin{equation}
\label{badlyapproximable3}
\omega_n = \xi(G^n(x)) \leq K \phi(\gamma^n)
\end{equation}
for all $n$ sufficiently large, where $\gamma = 1 + \lceil e^E\rceil$. By increasing $K$, we may ensure that (\ref{badlyapproximable3}) holds for all $n\in\N$.

Thus, we are done if we show that the set of $x$ for which there exists $K$ such that (\ref{badlyapproximable3}) holds for all $n\in\N$ is a null set. Given $n\in\N$ and $K > 0$ let
\[
S_{\psi,n,K} = \{x\in J_I: \text{(\ref{badlyapproximable3}) holds for
  $n,K$}\}
\]
and 
\[
S_{\psi,n,K}^+ = \bigcap_{j=0}^{n - 1} S_{\psi,j,K}.
\]
To complete the proof of Theorem~\ref{theoremconversestrong} we must therefore show that
\begin{equation}
\label{SpsiinftyK}
m_I\left(S_{\psi,\infty,K}^+\right) = 0 \all K > 0.
\end{equation}
\noindent Fix $K > 0$.

For each $n\in\N$, let $k_n = K\phi(\gamma^n)$. In the notation of Lemma~\ref{lemmaSomega}, we have
\begin{align*}
S_{\psi,n,K} &= \bigcup_{\omega\in A_n}S_\omega\\
S_{\psi,n + 1,K} &= \bigcup_{\omega\in A_n}S_{\omega,k_n}^+,
\end{align*}
where
\[
A_n = \prod_{j = 0}^{n - 1}\{1,\ldots,k_j\}.
\]
It therefore follows from (\ref{220120504}) that
\[
\frac{m_I(S_{\psi,n + 1}^+)}{m_I(S_{\psi,n}^+)}
\leq 1 - \frac{1}{4^{h_I}}\sum_{\substack{i\in I \\ i > k_n}}i^{-2 h_I}.
\]
On the other hand, by the implication (b1)$\Rightarrow$(c3) of Theorem~\ref{theoremcombinatorialcharacterization} we have 
\begin{equation}
\label{sigmadef}
\sum_{\substack{i \in I \\ i > k_n}} i^{-2h} \asymp k_n^{-h} \asymp \phi(\gamma^n)^{-h}.
\end{equation}
Thus for some constant $K_2 > 0$ depending on $K$, we have
\[
\frac{m_I(S_{\psi,n + 1,K}^+)}{m_I(S_{\psi,n,K}^+)}
\leq 1 - K_2\phi(\gamma^n)^{-h}.
\]
Thus
\[
m_I(S_{\psi,\infty,K}^+) \leq \prod_{n = 0}^\infty \left(1-K_2\phi(\gamma^n)^{-h}\right)
\]
which is zero if the series
\begin{equation}
\label{sigmaseries}
\sum_{n = 0}^\infty \phi(\gamma^n)^{-h}
\end{equation}
diverges. Now, by Cauchy's condensation test, (\ref{sigmaseries}) diverges if
and only if (\ref{weiss}) diverges. This demonstrates (\ref{SpsiinftyK}), completing the proof.

\begin{comment}
We now claim that the divergence of (\ref{weiss}) implies the divergence of (\ref{sigmaseries}), which will complete the proof. First of all, we can split up the series (\ref{weiss}) as
\[
\sum_{q = 1}^\infty \frac{1}{q\phi(q)^\alpha} = \sum_{n = 1}^\infty \sum_{q = Q_n}^{Q_{n + 1} - 1}\frac{1}{q\phi(q)^\alpha}
\]
where $Q_n := \lceil \gamma_2^n\rceil$. Since $\phi$ is increasing we have
\[
\frac{1}{q\phi(q)^\alpha} \leq \frac{1}{\gamma_2^n \phi(\gamma_2^n)^\alpha}
\]
for all $q\geq Q_n \geq \gamma_2^n$. In particular
\begin{align*}
\sum_{q = 1}^\infty \frac{1}{q\phi(q)^\alpha} &\geq \sum_{n = 1}^\infty \sum_{q = Q_{n - 1}}^{Q_n - 1}\frac{1}{\gamma_2^n\phi(\gamma_2^n)^\alpha}\\
&= \sum_{n = 1}^\infty \frac{Q_n - Q_{n - 1}}{\gamma_2^n\phi(\gamma_2^n)^\alpha}\\
&\asymp \sum_{n = 1}^\infty \frac{1}{\phi(\gamma_2^n)^\alpha}.
\end{align*}
\end{comment}

\section{Proof of Theorem~\ref{theoremexistence}}\label{sectiontheoremexistence}
In this section we will prove the following theorem:
\begin{theorem}
\label{theoremexistence}
For every $0 < \delta \leq 1$ there exists an infinite set $I\subset\N$ such that $\HD(J_I) = \delta$ and such that $\HH^\delta\given_{J_I}$ is Ahlfors $\delta$-regular.
\end{theorem}
Fix $0 < \delta \leq 1$. If $\delta = 1$, we let $I = \N$; the conclusion of the proposition is satisfied since then $\HH^\delta\given_{J_I}$ is simply Lebesgue measure. Thus, we shall assume without loss of generality that $\delta < 1$. We observe that by the implication (c1)$\Rightarrow$(b1) of Theorem~\ref{theoremcombinatorialcharacterization}, to prove Theorem~\ref{theoremexistence} it suffices to find a set $I$ satisfying
\begin{equation}
\label{deltaequalsh}
\HD(J_I) = \delta
\end{equation}
and
\begin{equation}
\label{combinatoriallydeltaregular}
\#(B(y,r)\cap I)\asymp r^\delta.
\end{equation}
We will begin by finding a set $I_0$ which satisfies (\ref{combinatoriallydeltaregular}) but not necessarily (\ref{deltaequalsh}). Then we will construct a set $R$ which satisfies (\ref{deltaequalsh}) but not necessarily (\ref{combinatoriallydeltaregular}). Finally we will combine $I_0$ and $R$ into a single set $I_\delta$ which is satisfies both (\ref{deltaequalsh}) and (\ref{combinatoriallydeltaregular}).

\subsection{Constructing $I_0$}
\begin{lemma}
\label{lemmaI0}
There exists a set $I_0\subset\N$ satisfying \textup{(\ref{combinatoriallydeltaregular})}.
\end{lemma}
\begin{proof}
Indeed, let $I_0$ be the set of all sums of the form
\[
1 + \sum_{n\in\N}a_n\lfloor 2^{n/\delta}\rfloor,
\]
where $a_n = 0$ or $1$ for all $n\in\N$, with only finitely many $1$s. It is readily verified that $I_0$ satisfies (\ref{combinatoriallydeltaregular}).
\end{proof}

\subsection{Constructing $R$}
We define a sequence of subsets $R_N\subset\N$ by induction:
\begin{itemize}
\item[1.] Let $R_1 = \{1\}$.
\item[2.] Suppose that $R_{N - 1}\subset \{1,\ldots,N - 1\}$ has been
  defined for some $N\geq 2$. If 
\[
\lambda_\delta(R_{N - 1}\cup\{N\}) < 1,
\]
then let $R_N = R_{N - 1} \cup \{N\}$, otherwise let $R_N = R_{N - 1}$.
\end{itemize}
\begin{observation}
\label{observationlambdaRN}
For all $N\in\N$
\[
\lambda_\delta(R_N) < 1.
\]
\end{observation}
\begin{proof}
The base case follows either from direct computation or from Bowen's formula (Theorem~\ref{BowenFormula}); the inductive step follows from the construction of $R_N$.
\end{proof}

\begin{claim}
$R:= \bigcup_N R_N$ is not cofinite.
\end{claim}
\begin{proof}
By Theorem~\ref{continuityofpressure} and by the previous observation, we have $\lambda_\delta(R)\leq 1$. Combining with Bowen's formula, we see that $\HD(J_R) \leq \delta < 1 = \HD(J_\N)$. In particular, $R\neq\N$.

Thus if we suppose by contradiction that $R$ is cofinite, then $\N\butnot R$ has a maximal element $M$; moreover, we know that $M\geq 2$ since $1\in R$. But then by the construction of $R_M$, we have 
\[
\lambda_\delta(R_{M - 1}\cup\{M\}) \geq 1
\]
and so by Lemma~\ref{lemmaKZ} we have
\begin{equation}
\label{toomuchR}
\lambda_\delta(R_{M - 1}) \geq 1 - \left(\frac{2}{2 + M}\right)^{2\delta}\cdot
\end{equation}
On the other hand, by Observation~\ref{observationlambdaRN} we have
\[
\lambda_\delta(R_{M - 1}\cup\{M + 1,\ldots,N\}) = \lambda_\delta(R_N) < 1
\]
for every $N\in\N$. So, applying Lemma~\ref{lemmaKZ}, we see  that
\begin{align*}
\left(\frac{2}{2 + M}\right)^{2\delta}
&> \lambda_\delta(R_{M - 1}\cup\{M + 1,\ldots,N\}) - \lambda_\delta(R_{M - 1})\\
&\geq \sum_{i = M + 1}^N \left(\frac{1}{1 + i}\right)^{2\delta}\\
&> \int_{x = M + 1}^{N + 1}\left(\frac{1}{1 + x}\right)^{2\delta}\dee x.
\end{align*}
Since $N$ was arbitrary, we can take the limit as $N$ approaches
infinity and so we have 
\[
\int_{x = M + 1}^\infty \left(\frac{1}{1 + x}\right)^{2\delta}\dee x <
\left(\frac{2}{2 + M}\right)^{2\delta}\cdot 
\]
If $\delta\leq 1/2$, then the left hand integral diverges, a contradiction. If $\delta > 1/2$, the left hand integral converges and we have 
\[
\frac{(M + 2)^{1 - 2\delta}}{2\delta - 1} < \left(\frac{2}{2 + M}\right)^{2\delta}\cdot
\]
Rearranging yields
\[
M + 2 < 2^{2\delta}(2\delta - 1) \leq 2^2 (2 - 1) = 4.
\]
This contradicts $M\geq 2$ and the proof is finished.
\end{proof}
\begin{observation}
It follows from (\ref{toomuchR}), Observation~\ref{observationlambdaRN}, and Theorem~\ref{continuityofpressure} that
\[
\HD(J_R) = \delta.
\]
\end{observation}

\subsection{Combining $I_0$ and $R$}
Fix $N_1\in\N\butnot R$ large, to be determined later.\footnote{Specifically, we let $N_1$ be large enough so that (\ref{N1def}) cannot hold whenever $M\geq N_1$.} By the construction of $R_{N_1}$ we have (\ref{toomuchR}) with $M = N_1$, and so 
\[
1 - \left(\frac{2}{2 + N_1}\right)^{2\delta} \leq \lambda_\delta(R_{N_1 - 1}) < 1.
\]
Now let $I_0$ be as in Lemma~\ref{lemmaI0}, and let
\begin{align*}
I_+ &:= 2I_0\\
I_- &:= 2I_0 - 1.
\end{align*}
It is evident that any set $I_\delta\subset \N$ satisfying
\begin{equation}
\label{impliesstar}
I_- \subset_* I_\delta \subset_* I_+\cup I_-
\end{equation}
satisfies (\ref{combinatoriallydeltaregular}), where $A\subset_* B$ means $\#(A\butnot B) < \infty$. We will construct such a set
recursively. Now by the implication (c1)$\Rightarrow$(c3) of Theorem~\ref{theoremcombinatorialcharacterization}\footnote{Note that the implication holds even when $h\neq \HD(J_I)$.}, we have
\[
\sum_{i\in I_-}\left(\frac{2}{2 + i}\right)^{2\delta} \asymp \sum_{i\in I_-}i^{-2\delta} < \infty;
\]
thus we may choose $N_2$ large enough so that
\begin{equation}
\label{N1def}
\sum_{i\in I_-\butnot\{1,\ldots,N_2\}}\left(\frac{2}{2 + i}\right)^{2\delta} < 1 - \lambda_\delta(R_{N_1 - 1}). 
\end{equation}
We will now construct a sequence of sets $(I_N)_{N\geq N_1 - 1}$
recursively in the following manner: 
\begin{itemize}
\item[1.] Let
\[
I_{N_1 - 1} = R_{N_1 - 1}\cup(I_-\butnot\{1,\ldots,N_2\}).
\]
\item[2.] Suppose that $I_{N - 1}$ has been defined for some $N\geq
  N_1$. If $N\notin I_+\cup\{N_1\}$, then let $I_N = I_{N - 1}$. 
\item[3.] If $N\in I_+\cup\{N_1\}$, and if
\[
\lambda_\delta(I_{N - 1}\cup\{N\}) < 1,
\]
then let $I_N = I_{N - 1} \cup \{N\}$.
\item[4.] Otherwise, let $I_N = I_{N - 1}$.
\end{itemize}

\begin{observation}
For all $N\geq N_1-1$
\[
\lambda_\delta(I_N) < 1.
\]
\end{observation}
\begin{proof}
The base case of induction follows from Lemma~\ref{lemmaKZ} together with
(\ref{N1def}). The induction step follows from the construction of
$I_N$. 
\end{proof}

\begin{claim}
Case 4 occurs infinitely many times.
\end{claim}
\begin{proof}
As $N_1\notin R$, we know that Case 4 occurs at least once, namely at $N = N_1$. If we suppose by contradiction that it occurs only finitely often, then there is some maximal value $M$ at which it occurs. In particular 
\[
\lambda_\delta(I_{M - 1}\cup\{M\})\geq 1,
\]
and applying Lemma~\ref{lemmaKZ} gives
\begin{equation}
\label{toomuchI}
\lambda_\delta(I_{M - 1}) \geq 1 - \left(\frac{2}{2 + M}\right)^{2\delta}\cdot
\end{equation}
On the other hand, by the above observation and by the maximality of $M$ we have
\[
\lambda_\delta(I_{M - 1}\cup(I_+\cap\{M + 1,\ldots,N\})) < 1
\]
for all $N\in\N$. Combining these last two formulas and then applying Lemma~\ref{lemmaKZ}, we see that
\begin{align*}
\left(\frac{2}{2 + M}\right)^{2\delta}
&> \lambda_\delta(I_{M - 1}\cup(I_+\cap\{M + 1,\ldots,N\})) -
\lambda_\delta(I_{M - 1})\\ 
&\geq \sum_{\substack{i = M + 1\\ i\in I_+}}^N \left(\frac{1}{1 + i}\right)^{2\delta}\cdot
\end{align*}
Since $N$ was arbitrary, we can take the limit as $N$ approaches infinity which yields
\begin{equation}
\label{N1def}
M^{-2\delta}
\asymp \left(\frac{2}{2 + M}\right)^{2\delta}
> \sum_{\substack{i = M + 1\\ i\in I_+}}^\infty \left(\frac{1}{1 + i}\right)^{2\delta}\\
\asymp M^{-\delta}.
\end{equation}
Since $\delta > 0$, this is a contradiction if $M$ is sufficiently large. Thus, if we let $N_1$ be large enough so that (\ref{N1def}) cannot hold whenever $M\geq N_1$, then this completes the proof of the claim. 
\end{proof}

Now let 
\[
I = I_\delta = \bigcup_{N\geq N_1 - 1}I_N.
\]
As mentioned earlier, it is clear that $I$ satisfies (\ref{combinatoriallydeltaregular}) since it satisfies (\ref{impliesstar}). Thus to complete the proof of Theorem~\ref{theoremexistence}, it suffices to demonstrate (\ref{deltaequalsh}). To this end, let $(M_k)_k$ be an increasing sequence of points at which Case 4 occurs. For each $k\in\N$, we have (\ref{toomuchI}) with $M = M_k$ i.e. 
\[
1 - \left(\frac{2}{2 + M_k}\right)^{2\delta} \leq \lambda_\delta(I_{M_k - 1}) < 1.
\]
Taking the limit as $k$ approaches infinity, we see that $\lambda_\delta(I) = 1$. Thus by Bowen's formula (Theorem~\ref{BowenFormula}), we have $\HD(J_I) = \delta$. This completes the proof.

\end{document}